\documentclass{amsart}

\usepackage{amsmath,amssymb,amsthm}
\usepackage{todonotes}
\usepackage{hyperref}
\usepackage{url}
\usepackage{lmodern}

\theoremstyle{plain}
\newtheorem{theorem}{Theorem}[section]
\newtheorem{conjecture}[theorem]{Conjecture}

\newtheorem{corollary}[theorem]{Corollary}

\theoremstyle{definition}

\newtheorem{question}[theorem]{Question}

\renewcommand{\bold}[1]{\medskip \noindent {\bf #1 }\nopagebreak}

\numberwithin{equation}{section}
\def\bC{\mathbb{C}}

\def\bQ{\mathbb{Q}}
\def\bR{\mathbb{R}}
\def\bZ{\mathbb{Z}}

\def\cH{\mathcal{H}}
\def\cM{\mathcal{M}}

\def\cN{\mathcal{N}}
\def\cQ{\mathcal{Q}}

\def\bk{\mathbf{k}}

\def\GL{\operatorname{GL}}
\def\Re{\operatorname{Re}}
\def\Im{\operatorname{Im}}

\def\phi{\varphi}

\title{A new orbit closure in genus 8?}
\author{Vincent Delecroix, Julian R\"uth}
\date{}

\begin{document}

 \begin{abstract}
We provide numerical evidence that the orbit closure of the unfolding of the
$(3,4,13)$-triangle is a previously unknown 4-dimensional variety.
 \end{abstract}

\maketitle

%%%%%%%%%%%%%%%%%%%
% TABLE OF CONTENTS
%%%%%%%%%%%%%%%%%%%
% allows subsections (depth 1) to be displayed in table of contents
\setcounter{tocdepth}{1}
\tableofcontents

%%%%%%%%%%%%%%%%%%%
%%%%%%%%%%%%%%%%%%%
\section{Introduction}
%%%%%%%%%%%%%%%%%%%
%%%%%%%%%%%%%%%%%%%

Given a polygon $P$ in $\bR^2$, we consider the associated billiard flow on
its unit tangent bundle $T^1 P$. When the polygon is \emph{rational}, that
is all its angles are rational multiple of $\pi$, then the 3-dimensional
phase space $T^1 P$ decomposes into a one-parameter family of invariant
surfaces that are the directional flows associated to the vector
fields $\omega = e^{i \theta}$ of an Abelian differential $\omega$ on
a compact Riemann surface $X$. See the survey \cite{MT} for this construction, which is
called the \emph{unfolding} of the billiard table $P$.

Let $\cM_{g,n}$ denote the moduli space of compact Riemann surface with
$n$ distinct marked points. Given a tuple of non-negative numbers $\kappa = (\kappa_1, \kappa_2, \ldots, \kappa_n)$
which sum to $2g - 2$,
we define the \emph{stratum} $\cH_g(\kappa)$
as the moduli space of tuples $(X, p_1, \ldots, p_n, \omega)$ such that $(X, p_1, \ldots, p_n) \in \cM_{g,n}$
and $\omega$ is a non-zero Abelian differential on $X$ with zeros of order $\kappa_i$
at $p_i$. The condition on the sum of $\kappa$ implies that $\omega$ does not vanish on the
complement of $\{p_1, \ldots, p_n\}$ in $X$.
The \emph{local period coordinates map} which, to $(X, p_1, \ldots, p_n, \omega)$ associates
$$[\omega] \in H^1(X, \{p_1, \ldots, p_n\}; \bC)\simeq \bC^{2g+n-1}\simeq (\bR^2)^{2g+n-1},$$ provides each stratum with an atlas of local charts with transition maps in $\GL_{2g+n-1}(\bZ)$.

Each stratum $\cH_g(\kappa)$ admits a $\GL_2^+(\bR)$-action via the action on
periods.
This $\GL_2^+(\bR)$-action acts as a
renormalization for the directional flows and many results about the
combinatorics and the dynamics of billiards in rational polygons are
obtained by studying properties of this action. See the
surveys~\cite{Esur,Fsur,HSsur,Masur,Wsurvey}.

We call a closed, irreducible $\GL_2^+(\bR)$-invariant subvariety of a stratum of Abelian differentials
an \emph{invariant subvariety}. By~\cite{EM,EMM,Fi1}, any closed $\GL_2^+(\bR)$-invariant
subset of a stratum is finite union of invariant subvarieties. Moreover, \cite{EM, EMM} gives that the image of any invariant $\cN$ under the local period coordinates map is locally identified with a finite union of linear subspaces of $H^1(X, \{p_1, \ldots, p_n\}; \bC)$ of the form
$V \oplus \sqrt{-1} V$ for some $V \subset H^1(X, \{p_1, \ldots, p_n\}; \bR)$; at typical points of $\cN$, only a single subspace is required.

Recall the following definitions:  the \emph{rank} of $\cN$ is half
the dimension of the image of $V$ in absolute cohomology $H^1(X; \bR)$; the \emph{rel} of $\cN$ is its (complex) dimension minus two times its rank; and the \emph{field of definition}\footnote{This is not the same as the field of definition of $\cN$ in the sense of algebraic geometry.} $\bk(\cN)$ is the smallest field over which  the linear subspace $V$ is defined \cite{Wfield}. Here it is helpful to recall that the image of $V$ in $H^1(X, \bR)$ is symplectic \cite{AEM}, and that real cohomology has a canonical integral structure.

%An important family of invariant subvarieties are obtained from quadratic
%differentials with at most simple poles.
Similarly to the Abelian case,
we define strata of quadratic differentials $\cQ_g(\mu)$.
% where $\mu = (\mu_1, \ldots, \mu_n)$ is a tuple of integers $\geq -1$ such that $\mu_1 + \cdots + \mu_n = 4g-4$.
By a
standard  construction, a quadratic differential admits a double cover
on which the quadratic differential lifts to the square of an Abelian
differential. We denote by $\widetilde{\cQ}_g(\mu)$ the locus of all such double
covers\footnote{Note that taking the double cover of a quadratic differential $(X, Q)$ does not specify the points $\tilde{x}_1$, \ldots, $\tilde{x}_{n'}$ in the cover. In our situation we only consider coverings of Abelian differentials in the stratum $\cQ_4(11,1)$. Each double cover has two zeros (of order $12$ and $2$) and and we label them in the same order as in the base surface.} of elements of $\cQ_g(\mu)$.

We now state our main result. By default, all dimensions are over $\bC$.
\begin{conjecture}\label{conj:main}
Let $\cN \subset \widetilde{\cQ}_4(11,1) \subset \cH_8(12,2)$ $\cN$ be the $\GL_2(\bR)$-orbit closure of the unfolding of the $(3,4,13)$-triangle.
Then $\cN$ is 4-dimensional and has field of definition $\bk(\cN)=\bQ[\sqrt{5}]$.
\end{conjecture}

As a partial result, we provide a computer assisted proof of the lower bound.
\begin{theorem} \label{thm:lowerBound}
The invariant subvariety $\cN$ is at least 4-dimensional and is defined over either $\bQ[\sqrt{5}]$ or $\bQ$. Moreover, it contains all unfoldings of quadrilaterals with angles $\{ \frac{2\pi}{10}, \frac{2\pi}{10}, \frac{3\pi}{10}, \frac{13\pi}{10}\}.$
\end{theorem}

Rational billiards whose unfoldings have non-dense $\GL_2^+(\bR)$-orbit closures are rare. We
refer the reader to Section~\ref{ssec:billiard:orbit:closures} for a survey of the
literature and experimental results about triangles and quadrilaterals, emphasizing for now that Conjecture~\ref{conj:main} adds to recent discoveries of the same flavor by McMullen-Mukamel-Wright and Eskin-McMullen-Mukamel-Wright \cite{MMW, EMMW}.

Let us point out a few properties of $\cN$. Since $\widetilde\cQ_4(11,1)$ has no rel, it follows that $\cN$ has no rel. Hence, if Conjecture~\ref{conj:main} holds, the fact that $\cN$ is 4-dimensional implies that it has rank 2. We also note that \cite[Remark 3.4]{ApisaWright} implies that unfoldings of typical $( \frac{2\pi}{10}, \frac{2\pi}{10}, \frac{3\pi}{10}, \frac{13\pi}{10})$ quadrilaterals do not cover lower genus translation surfaces, so $\cN$ is primitive. Finally, we point out as follows that $\cN$ is the orbit closure of many unfoldings.

\begin{corollary} \label{C:3.4.13}
Assuming Conjecture~\ref{conj:main}, for all but finitely many similarity classes of quadrilaterals with angles $\{ \frac{2\pi}{10}, \frac{2\pi}{10}, \frac{3\pi}{10}, \frac{13\pi}{10}\}$, the orbit closure of the unfolding is equal to $\cN$.
\end{corollary}

\begin{proof}[Proof of Corollary.]
It follows from~\cite{AEM} that the orbit of any surface in $\cN$ is closed or dense. It follows from~\cite{EFW} that $\cN$ has only finitely many closed orbits. This gives the first statement. (The same result with countably rather than finitely many exceptions follows from remarks in~\cite{MirWri2}.)
\end{proof}

Section~\ref{sec:explicit} gives examples of surfaces in $\cN$, presented via polygons in a particularly nice way, together with the linear equations locally defining $\cN$ at these surfaces. It contains the proof of Theorem~\ref{thm:lowerBound}. Section~\ref{sec:flatsurf} discusses the software suite \texttt{flatsurf} that is used throughout the proof in this article.

\bold{Acknowledgments.} We thank Alex Eskin for assisting in the development of the program discussed in Section~\ref{sec:flatsurf}, and for helpful conversations.
We thank Alex Wright for continuous support and for many insights on how one may approach Conjecture~\ref{conj:main}.

Julian R\"uth was supported by the Simons Foundation Simons Investigator grant of Alex Eskin. Alex Wright was partially supported by  NSF Grant DMS 1856155 and a Sloan Research Fellowship.

\section{Explicit surfaces in $\cN$}\label{sec:explicit}
By construction, there is an invariant subvariety $\overline{\cN}\subset \cQ_4(11,1)$ such that $\cN$ is the locus of double covers of quadratic differentials in $\overline{\cN}$. In this section we give especially nice polygonal presentations for some surfaces in $\overline{\cN}$ and provide a 4-dimensional linear subspace of the tangent space to $\cQ_4(11,1)$ that $\cN$ must contain. This linear subspace is described by equations near these particularly nice surfaces.

In order to describe linear equations in relative cohomology one needs to pick
a specific surface in $\overline{\cN}$ and a specific homology basis. Our choice
is shown in Figure~\ref{fig:layout}.
\begin{figure}[h]
    \centering
    {
      \def\svgwidth{\columnwidth}
      \fontfamily{lmss}\fontsize{5}{5}\selectfont
      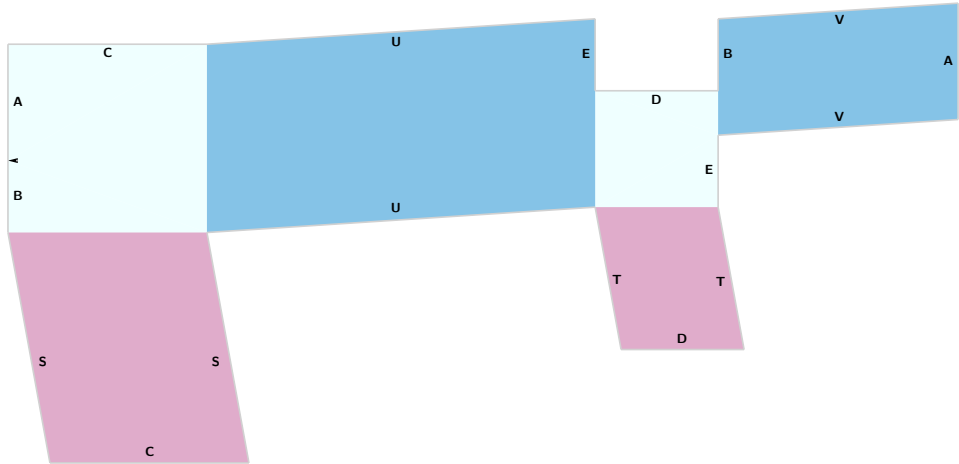
    }
\caption{A quadratic differential in $\overline{\cN} \subset \cQ_4(11,1)$ obtained from the gluing of 6 polygons. The sides $B$, $C$, $D$, $S$, $T$, $U$ and $V$ are identified by translations $z \mapsto z + c$ while the sides $A$ and $E$ are identified by "half translations" $z \mapsto -z + c$, respectively.}
\label{fig:layout}
\end{figure}

The complex dimension of $\cQ_4(11, 1)$ is 8, so our (local) linear subvariety contained in $\overline{\cN}$ is locally described by 4 equations.
On the surface illustrated, these equations become particularly elegant. In the theorem below, we write the equations directly
in the homology of the surface that supports the quadratic differential even though these equations live in the
double cover. The integration of the square root of a quadratic differential is only well defined up to sign,
but Figure~\ref{fig:layout} makes it clear which sign to choose.
\begin{theorem}
\label{thm:equations}
The invariant subvariety contains the following (local) linear subvariety.
Given the relative homology basis $\{A, B, C, D, E, S, T, U, V\}$ for the surface
depicted in Figure~\ref{fig:layout}, a set of defining equations for the linear subvariety
is given by
\begin{align*}
\int_A \sqrt{q} = \pm \phi \cdot \int_B \sqrt{q} \qquad & \qquad \int_C \sqrt{q} = \pm \phi \cdot \int_D \sqrt{q} \\
\int_S \sqrt{q} = \pm \phi \cdot \int_T \sqrt{q} \qquad & \qquad \int_U \sqrt{q} = \pm \phi \cdot \int_V \sqrt{q}
\end{align*}
where $\phi = \frac{1 + \sqrt{5}}{2}$.
\end{theorem}

\begin{proof}
The proof is a rather straightforward computation with the flatsurf software suite that is described in Section~\ref{sec:flatsurf}. All the steps of the computation are explained in a Jupyter notebook available at \url{https://git.io/Ju2lU}. Here we only provide a sketch of the steps and leave all computational details to the software suite (or the determined reader).

First, starting from the unfolding of the triangle with angles $( \frac{3\pi}{20}, \frac{4\pi}{20}, \frac{13\pi}{20})$, we compute tangent vectors to the orbit closure, see Section~\ref{sec:flatsurf}. In this case, we find four tangent vectors. We hence obtain a four-dimensional subspace contained in $\overline{\cN}$. Next, we change our base point by following the available four-dimensional space and arrive at the surface presented in Figure~\ref{fig:layout}. The tangent space is then expressed in the new homology basis shown in Figure~\ref{fig:layout}.
\end{proof}

\begin{proof}[Proof of Theorem~\ref{thm:lowerBound}]
The lower bound on the dimension follows immedialtly from Theorem~\ref{thm:equations}. For the field of definition, it must be contained in the holonomy field of any surface contained in $\cN$. The four-dimensional subspace in Theorem~\ref{thm:equations} provides surfaces defined over $\bQ[\sqrt{5}]$. It is hence either $\bQ$ or $\bQ[\sqrt{5}]$.
\end{proof}

\section{Exploring $\GL_2(\bR)$-orbit closures with \texttt{flatsurf}}\label{sec:flatsurf}
Conjecture~\ref{conj:main} was made using by the \texttt{flatsurf} suite. It is partly
inspired by Alex Eskin's \texttt{polygon} program. This free software suite is maintained by Vincent
Delecroix and Julian Rüth. It consists of the C, C++ and Python libraries
\begin{itemize}
\item \texttt{e-antic}~\cite{e-antic} and \texttt{exact-real}~\cite{exact-real} for computations with exact real numbers,
\item \texttt{intervalxt}~\cite{intervalxt} to work with interval exchange transformations,
\item \texttt{flatsurf}~\cite{flatsurf}, \texttt{sage-flatsurf}~\cite{sage-flatsurf} for translation surfaces, \texttt{surface-dynamics}~\cite{surface-dynamics}, \texttt{veerer}~\cite{veerer},
\item \texttt{vue-flatsurf}~\cite{vue-flatsurf} and \texttt{ipyvue-flatsurf}~\cite{ipyvue-flatsurf} for visualization.
\end{itemize}
All C and C++ libraries have Python interfaces, mostly through
\texttt{cppyy}~\cite{cppyy} which provides automatic Python bindings for C++ libraries.

We now outline how this software suite rigorously gives ``lower bounds" for the size of orbit closures; a more in-depth
discussion will be in~\cite{flatsurf-article, boshernitzan-article}.
Concrete examples of such computations in Python are available in the
\texttt{sage-flatsurf} documentation at
\url{https://flatsurf.github.io/sage-flatsurf/}.

An invariant subvariety $\cM$ in a stratum of Abelian differentials is completely
determined by a translation surface $(X, \omega)$ and a vector subspace $V$
of $H^1(X, Z(\omega); \bR)$ where $Z(\omega)$ denotes the set of zeros of $\omega$. This vector subspace is
such that $V \oplus \sqrt{-1}\,V$
can be identified with the tangent space of $\cM$ at $(X, \omega)$ via the period map
$\omega \mapsto [\omega] \in H^1(X, Z(\omega); \bC)$.
Our approach to compute the orbit closure consists of building a sequence of
subspaces $V_0 \subset V_1 \subset \ldots$ of $V$. We start with the
\emph{tautological space} $V_0 = \bR \Re(\omega) \oplus \bR \Im(\omega)$.
By the $\GL_2(\bR)$-invariance
the two vectors $\Re(\omega)$ and $\Im(\omega)$ belong to $V$ and so $V_0\subset V$.
We then apply the following procedure to obtain $V_{i+1}$ from $V_i$.
\begin{enumerate}
\item Find a direction in $(X, \omega)$ that contains a cylinder
and apply the cylinder deformation theorem~\cite{Wcyl} to get tangent vectors
in the orbit closure.
\item Set $V_{i+1}$ to the sum of $V_i$ and the newly discovered tangent
vectors.
\end{enumerate}
\begin{figure}[h]
    \centering
    {
      \def\svgwidth{\columnwidth}
      \fontfamily{lmss}\fontsize{4}{4}\selectfont
      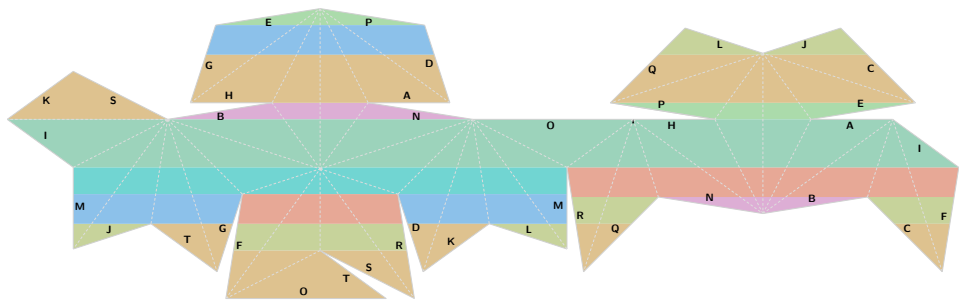
    }
    \caption{This unfolding of the triangle with angles $(\frac{3\pi}{20},
    \frac{4\pi}{20}, \frac{13\pi}{20})$ decomposes into 8 cylinders in
    horizontal direction. Application of the cylinder deformation theorem to this
    decomposition does, however, not produce any new tangent vectors that were
    not already in $V_0$.}
\end{figure}
Of course, there is an issue with this approach: the algorithm has
no stopping time, unless $V_n = H^1(X, Z(\omega); \bR)$ for some
$n$, in which case the orbit closure is dense. At each step,
it is guaranteed that $V_n \oplus \sqrt{-1}\, V_n$ is
a subspace of the tangent space but we have no way yet to automatically
certify that a candidate subspace is indeed the full tangent space.
An even more serious theoretical issue is the following question.
\begin{question} \label{Q:termination-no-deformation}
If we consider all directions that contain cylinders, do we ever get
$V_n=V$? That is, does the cylinder deformation theorem eventually discover the
entire tangent space?
\end{question}

P.~Apisa communicated to us some evidence that Question~\ref{Q:termination-no-deformation}
may have a negative answer. More specifically, during a summer school at ICERM,
 his students B.~Harper, H.~Wan, and H.~Yang discovered a specific
%$(5, 3, 3, 1)$-quadrilateral (see Section~\ref{ssec:billiard:orbit:closures})
quadrilateral
whose
unfolding seems to have $\dim V_n \leq 3$ for all $n$ even though the $\GL_2(\bR)$-orbit
is known to be dense in its 6-dimensional stratum (so $\dim V=6$).

The following modification of the algorithm turns out to be useful in
practice to make the algorithm more efficient, and it may additionally circumvent a negative answer to Question~\ref{Q:termination-no-deformation} in some cases. At step $i$ of the algorithm,
instead at looking for cylinders in $(X, \omega)$ one can look for
cylinders in a deformation of $(X, \omega)$ along $V_i + \sqrt{-1} V_i$.

\begin{question} \label{Q:termination-deformation}
When deformations are performed in each step of the algorithm, do we ever get
$V_n = V$?
\end{question}

Compare Questions~\ref{Q:termination-no-deformation} and~\ref{Q:termination-deformation} to \cite[Question 5]{SWprobs},  \cite[Remark 6.1]{MirWri}, and \cite[Theorem 1.1]{LNW}, and note in particular that if there exists a completely parabolic surface that is not Veech, both questions will have a negative answer.

\section{Exceptional billiards}
\label{ssec:billiard:orbit:closures}
We now survey the various rational billiards whose unfoldings have non-dense
orbit closures. Most of them were already known in the literature. In addition to the
triangle with angles $( \frac{3\pi}{20}, \frac{4\pi}{20}, \frac{13\pi}{20})$,
the flatsurf suite allowed us to find two other exceptional
orbit closures. We conjecture that the finite list of exceptions in
Theorem~\ref{thm:exceptional:triangles} and Theorem~\ref{thm:exceptional:quadrilaterals}
is complete. Partial results in this direction are obtained in~\cite{MirWri2, LNZ, Apisa:triangles}.

We say that an $n$-tuple of positive integers $(a_1, \ldots,
a_n)$ is \emph{admissible} if
\begin{itemize}
\item $a_1 \leq a_2 \leq \ldots \leq a_n$,
\item $\gcd(a_1, \ldots, a_n) = 1$, and
\item $(n - 2) a_i \not= a_1 + \ldots + a_n$ for all $i \in \{1, \ldots, n\}$.
\end{itemize}
To the tuple $(a_1, \ldots, a_n)$ we associate the moduli space of flat metrics on
the sphere with conical singularities of angles $(a_1 \alpha, a_2 \alpha, \ldots, a_n \alpha)$ where $\alpha = \frac{(n-2)\,\pi}{a_1 + \ldots + a_n}$. Such locus of
flat metrics admits a cover of degree $a_1 + \ldots + a_n$ ramified over the
conical points which belongs to a stratum of Abelian or quadratic differentials
if $a_1 + \ldots + a_n$ is even or odd, respectively.

In this notation, this article considers the billiard in the
$(3, 4, 13)$ triangle and the $(2, 2, 3, 13)$ quadrilateral. Namely,
the $\GL_2(\bR)$-orbit closures of the unfoldings of these two
polygons is the invariant subvariety $\cN$ appearing in
Conjecture~\ref{conj:main} and Corollary~\ref{C:3.4.13}.

We say that an admissible tuple $(a_1, \ldots, a_n)$ is \emph{reducible} if
it satisfies one of the following conditions
\begin{itemize}
\item $n=3$ and $a_1 = a_2 < a_3$ or $a_1 < a_2 = a_3$, or
\item $n=4$ and $a_1 = a_2 < a_3 = a_4$.
\end{itemize}
The reducible
tuples correspond to triangles and quadrilaterals
that admit a reflection symmetry (independently of the side lengths).
In particular their unfoldings are double covers of unfoldings of their
``halves". For example
\begin{itemize}
\item the isosceles $(2, 2, 3)$-triangle unfolds to $\cH_3(2, 1^2)$ while its half, the right $(3, 4, 7)$-triangle unfolds to $\cQ_0(2, 1, -1^7)$,
\item the isosceles $(3, 3, 5)$-triangle unfolds to $\cH_5(4, 2^2)$ while its half, the right $(5, 6, 11)$-triangle unfolds to $\cQ_0(4, 3, -1^11)$,
\item the isosceles $(3, 3, 4)$-triangle unfolds to $\cQ_2(2, 1^2)$ while its half, the right $(2, 3, 5)$-triangle unfolds to $\cQ_0(1, -1^5)$.
\end{itemize}
An admissible tuple which is not reducible is called \emph{reduced}. An analysis using the \texttt{flatsurf} suite has rigorously proven the following.

\begin{theorem}
\label{thm:exceptional:triangles}
All admissible and reduced triples $(a,b,c)$ with $a + b + c \leq 58$ with
the exception of the list below correspond to triangular billiards whose unfoldings
have dense $\GL_2(\bR)$-orbit closures in their ambient strata.
\begin{itemize}
\item $\{(1, 2, 2k+1)\}_{k \geq 2}$, $\{(1, k, k+1)\}_{k \geq 3}$, $\{(2, 2k+1, 2k+3)\}_{k \geq 1}$: Teichmüller curves~\cite{V}, \cite{Vo}, \cite{W},
\item $(1, 4, 7)$, $(2, 3, 4)$, $(3, 4, 5)$, $(3, 5, 7)$: Teichmüller curves~\cite{Hooper}, \cite{KS}, \cite{V}, \cite{Vo},
\item $(1,3,6)$, $(1,3,8)$: rank-one but not Teichmüller curves (eigenform
loci inside a Prym subvariety)~\cite{AuAvDe},
\item $(1,4,11)$, $(1,4,15)$, $(3,4,13)$: rank 2 orbit closures~\cite{EMMW} and this article.
\end{itemize}
\end{theorem}
The $(1,3,6)$ and $(1,3,8)$ triangles were discovered as part of the exhaustive search
of the \texttt{flatsurf} suite on small triangles.

In~\cite{EMMW} the authors exhibited 6 quadrilaterals with rank 2 orbit closures and this article
adds one to them. This gives the following list of exceptional quadrilaterals
\begin{equation} \label{eq:ExceptionalQuadrilaterals}
\begin{array}{l}
(1,1,1,7), (1,1,1,9), (1,1,2,8), (1,1,2,12) \\ (1,2,2,11), (1,2,2,15), (2,2,3,13). \\
\end{array}
\end{equation}
\begin{theorem}
\label{thm:exceptional:quadrilaterals}
All admissible and reduced quadruples $(a,b,c,d)$ with $a + b + c + d \leq 32$ with
the exception of the list~\eqref{eq:ExceptionalQuadrilaterals} correspond to billiards in quadrilaterals whose
unfoldings have dense $\GL_2(\bR)$-orbit closures in their ambient strata.
\end{theorem}

\bibliography{mybib}{}
\bibliographystyle{amsalpha}

\end{document}